\documentclass[11pt]{article}
\usepackage{amsmath}
\usepackage{dsfont}
\usepackage{mathrsfs}
\usepackage{amsmath,amssymb}
\usepackage{amsfonts}
\usepackage{hyperref}
\usepackage{amsthm}
\usepackage{graphicx}
\usepackage{subfigure}
\usepackage{xcolor}

\hfuzz=\maxdimen
\newtheorem{lemma}{Lemma}[section]

\newtheorem{theorem}[lemma]{Theorem}

\newtheorem{remark}{Remark}

\numberwithin{equation}{section}

\DeclareFixedFont{\Acknowledgment}{OT1}{cmr}{bx}{n}{14pt}
\begin{document}

\title{\bf The $1$-Yamabe equation on graph}
\author{Huabin Ge, Wenfeng Jiang}
\maketitle

\begin{abstract}
We study the following $1$-Yamabe equation on a connected finite graph
$$\Delta_1u+g\mathrm{Sgn}(u)=h|u|^{\alpha-1}\mathrm{Sgn}(u),$$
where $\Delta_1$ is the discrete $1$-Laplacian, $\alpha>1$ and $g, h>0$ are known.
We show that the above $1$-Yamabe equation always has a nontrivial solution $u\geq0$, $u\neq0$.
\end{abstract}

\setcounter{section}{-1}
\section{Introduction}
Let $(M^n,g)$ be a smooth $n\geq3$ dimensional Riemannian manifold with scalar curvature $R$. The well known smooth Yamabe problem asks if there exists a smooth Riemannian metric $\tilde{g}$ that conformal to $g$ and has constant scalar curvature $\tilde{R}$. This leads to the consideration of the following smooth Yamabe equation
\begin{equation*}
\Delta u-\frac{n-2}{4(n-1)}Ru+\frac{n-2}{4(n-1)}\tilde{R}u^{\frac{n+2}{n-2}}=0,\;u>0.
\end{equation*}
If $u\in C^{\infty}(M)$ is a solution, then $\tilde{g}=u^{4/(n-2)}g$ has scalar curvature $\tilde{R}$. Today the Yamabe problem is completely understood thanks to basic contributions by Yamabe, Tr\"{u}dinger, Aubin and Schoen. We refer to \cite{Aubin}, \cite{MRS}-\cite{Y}, the survey \cite{Lee} and the references therein. In case $n=2$, the Yamabe problem reduces to the famous uniformization theorem, which asks for the consideration of the Kazdan-Warner equation
\begin{equation*}
\Delta u=K-\tilde{K}e^{2u}
\end{equation*}
If $u\in C^{\infty}(M)$ is a solution, then $\tilde{g}=e^{2u}g$ has Gaussian curvature $\tilde{K}$. We refer to \cite{KZ-2}-\cite{KZ-4} for the solvability of the smooth Kazdan-Warner equation.

In the series work \cite{GLY}-\cite{GLY3}, Grigor'yan-Lin-Yang observed that one can establish similar results on graphs. In \cite{GLY}, they studied the Kazdan-Warner equation $\Delta u=c-he^u$ on a finite graph $G$, and gave various conditions such that the equation has a solution. They characterize the solvability of the equation completely except for the critical case which was finally settled down by the first author Ge \cite{Ge-jmaa} of this paper. Ge \cite{Ge-arxiv}, Zhang-Chang \cite{Zhang-Chang-jmaa} then generalized their results to the $p$-th Kazdan-Warner equation $\Delta_p u=c-he^u$, where $\Delta_p$ is a type of discrete
$p$-Laplace operator, see (\ref{def-p-laplace}) in this paper for a definition. On infinite graphs which are analogues of noncompact manifolds, the absence of compactness indicates the difficulty to give a complete characterization of the solvability of the corresponding equation. Under some additional assumptions on the graphs and functions, Ge-Jiang \cite{Ge-Jiang-arxiv}, Keller-Schwarz \cite{Keller} get some results with totally different techniques and assumptions. In \cite{GLY'},
Grigor'yan-Lin-Yang studied the Yamabe type equation $-\Delta u-\alpha u=|u|^{p-2}u$ on a finite domain $\Omega$ of an infinite graph, with $u=0$ outside $\Omega$.
Ge \cite{Ge-pro-ams} studied the $p$-th Yamabe equation $\Delta_pu+hu^{p-1}=\lambda fu^{\alpha-1}$ on a finite graph under the assumption $\alpha\geq p>1$, which was generalized to $\alpha<p$ by Zhang-Lin \cite{ZhangLin-arxiv}. Ge-Jiang \cite{Ge-Jiang-CCM} further generalized Ge's results to get a global positive solution on an infinite graph under suitable assumptions.

In case $p>1$, the $p$-Laplace operator $\Delta_p$ exhibits more or less similar properties with the standard Laplace operator $\Delta$, both on the smooth manifolds and graphs. However, the $1$-Laplace operator $\Delta_1$ looks much different. The solutions for equations involving $\Delta$ and $\Delta_p$ are smooth, while those for $\Delta_1$ may be discontinuous. Since solutions in many interesting problems, e.g., in the signal processing and in the image processing, etc., may be discontinuous, the 1-Laplace operator $\Delta_1$ has been received much attention in recent years. In this paper, we study the $1$-th Yamabe equation on a finite graph. Different with the study of the $p$-th equations, variational methods can not be used here directly. The main idea of the proof is to take $\Delta_1$ as in some sense the limit of $\Delta_p$ as $p\rightarrow1$.

\section{Settings and main results}
Let $G=(V,E)$ be a finite graph, where $V$ denotes the vertex set and $E$ denotes the edge set. Fix a vertex measure $\mu:V\rightarrow(0,+\infty)$ and an edge measure $w:E\rightarrow(0,+\infty)$ on $G$. The edge measure $w$ is assumed to be symmetric, that is, $w_{xy}=w_{yx}$ for each edge $x\thicksim y$. Let $C(V)$ be the set of all real functions defined on $V$, then $C(V)$ is a finite dimensional linear space with the usual function additions and scalar multiplications. Professor Chang Kung-Ching (see \cite{ZGQ-JGraph}, or \cite{ZGQ-Advance}\cite{ZGQ-JComput}) introduced a $1$-Laplace operator on graphs as the following
\begin{equation}
\Delta_1f(x)=\frac{1}{\mu(x)}\sum\limits_{y\thicksim x}w_{xy}\mathrm{Sgn}(f(y)-f(x))
\end{equation}
for any $f\in C(V)$ and $x\in V$, in which
\begin{equation}\label{Def-Luo-normalized-Yamabe-flow}
\mathrm{Sgn}(t)=
\begin{cases}
\{1\}&\text{if $t>0$}\\
\{-1\}&\text{if $t<0$}\\
[-1,1]&\text{if $t=0$}
\end{cases}
\end{equation}
is a set valued function.
\begin{remark}
The addition of two subsets $A,B\subset \mathbb{R}^n$ is the set $\{x+y| x\in A, y\in B\}$, and for a scalar $\alpha$, the scalar multiplication $\alpha A$ is the set $\{\alpha x|x\in A\}$.
\end{remark}

In this paper, our main task is to study the following $1$-th Yamabe equation
\begin{equation}\label{def-1-Yamabe-equat}
\Delta_1u+g\mathrm{Sgn}(u)=h|u|^{\alpha-1}\mathrm{Sgn}(u),
\end{equation}
with $\alpha>1$, and $g,h\in C(V)$ be positive.

Note at each vertex $x\in V$, $\Delta_1f(x)$ is a subset of $\mathbb{R}$. For each function $u\in C(V)$, set
\begin{equation*}
A^+(x)=-\Delta_1u(x)+g(x)\mathrm{Sgn}(u(x))
\end{equation*}
and
\begin{equation*}
A^-(x)=h(x)|u(x)|^{\alpha-1}\mathrm{Sgn}(u(x)).
\end{equation*}
We say $u$ is a solution of the $1$-th Yamabe equation (\ref{def-1-Yamabe-equat}) if at each vertex $x\in V$,
$$A^+(x)\cap A^-(x)\neq\emptyset.$$

Obviously, if $u=0$ everywhere on $V$, both $A^+(x)$ and $A^-(x)$ contains $0$ as an element. Hence $u=0$ is always a trivial solution of the $1$-th Yamabe equation (\ref{def-1-Yamabe-equat}). We want to know if there is a nontrivial solution to (\ref{def-1-Yamabe-equat}). Our main result reads as
\begin{theorem}\label{main-thm-the-paper}
Let $\alpha>1$, $g,h\in C(V)$ be positive. The $1$-th Yamabe equation (\ref{def-1-Yamabe-equat}) has a nontrivial solution $u\geq0$, $u\neq0$.
\end{theorem}

For any $p>1$, the $p$-th discrete graph Laplace operator $\Delta_p:C(V)\rightarrow C(V)$ is
\begin{equation}\label{def-p-laplace}
\Delta_pf(x)=\frac{1}{\mu(x)}\sum\limits_{y\thicksim x}w_{xy}|f(y)-f(x)|^{p-2}(f(y)-f(x))
\end{equation}
for any $f\in C(V)$ and $x\in V$. If $f(y)=f(x)$, we require $|f(y)-f(x)|^{p-2}(f(y)-f(x))=0$.
$\Delta_p$ is a nonlinear operator when $p\neq2$ (see \cite{Ge-arxiv} for more properties about $\Delta_p$).

The main idea of the proof is to take $\Delta_1$ as in some sense the limit of $\Delta_p$ as $p\rightarrow1$. Hence we shall first establish an existence result of the positive solution $u_p$ to the following
$p$-th Yamabe equation
\begin{equation}\label{def-Yamabe-equ-p-1}
-\Delta_pu_p+gu_p^{p-1}=hu_p^{\alpha-1}
\end{equation}
for each $p>1$. By proving a uniform bound for $u_p$, and taking limit $p\rightarrow1$, we get a nontrivial solution to the $1$-th Yamabe equation (\ref{def-1-Yamabe-equat}).

It is remarkable that one can see other interesting phenomenons when $p\to 1$. There are many such works on smooth domains (see Kawohl-Fridman \cite{KF} and Kawohl-Schuricht \cite{KS}) and on graphs (see Chang \cite{ZGQ-JGraph,ZGQ-Advance,ZGQ-JComput}).\\

\noindent \textbf{Acknowledgements:} Both authors would like to thank Professor Gang Tian for constant encouragement. The research is supported by NNSF of China under Grant No.11501027.

\section{Sobolev embedding}
\label{sect-preliminary-lemma}

For any $f\in C(V)$, define the integral of $f$ over $V$ with respect to the vertex weight $\mu$ by
$$\int_Vfd\mu=\sum\limits_{x\in V}\mu(x)f(x).$$
Set $\mathrm{Vol}(G)=\int_Vd\mu$. Similarly, for any function $\psi$ defined on the edge set $E$, we define the integral of $\psi$ over $E$ with respect to the edge weight $w$ by
$$\int_E\psi dw=\sum\limits_{x\thicksim y}w_{xy}\psi_{xy}.$$
Specially, for any $f\in C(V)$,
$$\int_E|\nabla f|^pdw=\sum\limits_{x\thicksim y}w_{xy}|f(y)-f(x)|^p,$$
where $|\nabla f|$ is defined on the edge set $E$, and $|\nabla f|_{xy}=|f(y)-f(x)|$ for each edge $x\thicksim y$. Next we consider the Sobolev space $W^{1,\,p}$ on the graph $G$. Define
$$W^{1,\,p}(G)=\left\{f\in C(V):\int_E|\nabla f|^pdw+\int_V|f|^pd\mu<+\infty\right\},$$
and
$$\|f\|_{W^{1,\,p}(G)}=\left(\int_E|\nabla f|^pdw+\int_V|f|^pd\mu\right)^{\frac{1}{p}}.$$
Since $G$ is a finite graph, then
$W^{1,\,p}(G)$ is exactly $C(V)$, a finite dimensional linear space. This implies the following Sobolev embedding:

\begin{lemma}\label{lem-Sobolev-embedding}(Sobolev embedding)
Let $G=(V,E)$ be a finite graph. The Sobolev space $W^{1,\,p}(G)$ is pre-compact. Namely, if $\{\varphi_n\}$ is bounded in $W^{1,\,p}(G)$, then there exists some $\varphi\in W^{1,\,p}(G)$ such that up to a subsequence, $\varphi_n\rightarrow\varphi$ in $W^{1,\,p}(G)$.
\end{lemma}
\begin{remark}
The convergence in $W^{1,\,p}(G)$ is in fact pointwise convergence.
\end{remark}

\section{The $p$-th equation has a positive solution $u_p$}
Our first main technical observation is the following.
\begin{theorem}\label{thm-p-Yamabe-solvable}
Let $G$ be a finite connected graph. Assume $\alpha\geq p>1$, $g, h>0$. Then the following $p$-th Yamabe equation
\begin{equation}\label{def-Yamabe-equ-p-2}
-\Delta_pu+g|u|^{p-1}=\lambda h|u|^{\alpha-1}
\end{equation}
on $G$ always has a positive solution $u$ for some constant $\lambda>0$. Moreover, $u$ satisfies
\begin{equation*}
\int_Vhu^{\alpha}d\mu=1.
\end{equation*}
\end{theorem}

The first part of the theorem was established by Ge \cite{Ge-pro-ams}. The second part of the theorem is a consequence of the proof of the main theorem of \cite{Ge-pro-ams}. For completeness, we give a direct and great simplified proof here. It is remarkable that the condition $g>0$ is not needed so as (\ref{def-Yamabe-equ-p-2}) to have a positive solution. $g>0$ is used here (in fact, $g\geq0$ and $g\neq0$ is enough) to guarantee  that $\lambda>0$.

\begin{proof}
We minimize the nonnegative functional (note $g\geq0$)
\begin{equation*}
I(\varphi)=\int_E|\nabla \varphi|^pdw+\int_Vg\varphi^pd\mu
\end{equation*}
in the non-empty set
\begin{equation*}
\Gamma=\left\{\varphi\in W^{1,\,p}(G):\int_Vh\varphi^{\alpha} d\mu=1,\;\varphi\geq0\right\}.
\end{equation*}
Let
\begin{equation*}
\beta=\inf_{\varphi\in\Gamma} I(\varphi).
\end{equation*}
Take a sequence of functions $u_n\in\Gamma$ such that $I(u_n)\rightarrow\beta$. Obviously, $\{u_n\}$ is bounded in
$W^{1,\,p}(G)$. Therefore by the Sobolev embedding Lemma \ref{lem-Sobolev-embedding}, there exists some $u\in C(V)$ such that up to a subsequence, $u_n\rightarrow u$
in $W^{1,\,p}(G)$. We may well denote this subsequence as $u_n$. Because the set $\Gamma$ is closed, we see $u\in\Gamma$, that is, $u\geq0$, and
\begin{equation}\label{u-nonzero}
\int_Vhu^{\alpha}d\mu=1.
\end{equation}
Based on the method of Lagrange multipliers, one can consider the nonrestraint minimization of the following functional
\begin{equation*}
J(\varphi)=I(\varphi)+\gamma\left(\int_Vh\varphi^{\alpha}d\mu-1\right)
\end{equation*}
and calculate the Euler-Lagrange equation of $u$ as follows:
\begin{equation}\label{Lagrange-multip-methd}
-\Delta_pu+gu^{p-1}+\frac{\gamma\alpha}{p}hu^{\alpha-1}\geq0
\end{equation}
where $\gamma$ is a constant. (\ref{Lagrange-multip-methd}) implies $u>0$. In fact, note the graph $G$ is connected, if $u>0$ is not satisfied, since $u\geq0$ and not identically zero (this can be seen from (\ref{u-nonzero})), there is an edge $x\thicksim y$, such that $u(x)=0$, but $u(y)>0$. Now look at $\Delta_pu(x)$,
$$\Delta_pu(x)=\frac{1}{\mu(x)}\sum\limits_{z\thicksim x}w_{xz}|u(z)-u(x)|^{p-2}(u(z)-u(x))>0,$$
which contradicts (\ref{Lagrange-multip-methd}). Hence $u>0$ is in the interior of the space $\{\varphi:\varphi\geq0\}$ and hence then (\ref{Lagrange-multip-methd}) becomes an equality
\begin{equation}\label{u-equation}
-\Delta_pu+gu^{p-1}=\lambda hu^{\alpha-1},
\end{equation}
where $\lambda=-\frac{\gamma\alpha}{p}$. Multiplying $u$ at the two sides of (\ref{u-equation}) and integrating, we have
\begin{equation*}
\int_E|\nabla u|^pdw+\int_Vgu^{p}d\mu=\int_V(-u\Delta_pud\mu+gu^{p})d\mu=\lambda \int_Vhu^{\alpha}d\mu.
\end{equation*}
This leads to
\begin{equation*}
\lambda=\frac{\int_E|\nabla u|^pdw+\int_Vgu^{p}d\mu}{\int_Vhu^{\alpha}d\mu},
\end{equation*}
from which we see $\lambda>0$ and hence the conclusion.
\end{proof}

\section{The solutions $u_p$ are uniformly bounded}
For each $p\in(1,\alpha)$, let $u_p>0$ be a solution to the $p$-th Yamabe equation (\ref{def-Yamabe-equ-p-2}). Thus
\begin{equation}\label{u-p-solution}
-\Delta_pu_p+gu_p^{p-1}=\lambda_p hu_p^{\alpha-1},
\end{equation}
where
\begin{equation}\label{lamda-express}
\lambda_p=\frac{\int_E|\nabla u_p|^pdw+\int_Vgu_p^{p}d\mu}{\int_Vhu_p^{\alpha}d\mu}>0.
\end{equation}
Moreover,
\begin{equation}\label{u-constraint-1}
\int_Vhu_p^{\alpha}d\mu=1.
\end{equation}

In the following, we use $c(\alpha,h, G)$ as a constant depending only on the information of $\alpha$, $h$ and $G$, use $c(\alpha,g,h,G)$ as a constant depending only on the information of $\alpha$, $g$, $h$ and $G$. Note that the information of $G$ contains $V$, $E$, the vertex measure $\mu$ and the edge weight $w$. For any function $f\in C(V)$, we denote $f_m=\min_{x\in V}f(x)$ and $f_M=\max_{x\in V}f(x)$.

\begin{lemma}\label{thm-up-bound}
There are positive constants $c_1(\alpha,h,G)\geq1$ and $c_2(\alpha,h,G)\leq1$ so that
\begin{equation}\label{formula-up-bound}
c_2(\alpha,h,G)\leq\max_{x\in V}u_p(x)\leq c_1(\alpha,h,G).
\end{equation}
\end{lemma}
\begin{proof}
The above estimates come from (\ref{u-constraint-1}). For all $p\in(1,\alpha)$ and $x\in V$, by
$$h(x)u_p^{\alpha}(x)\mu(x)\leq\int_Vhu_p^{\alpha}d\mu=1,$$
we see $u_p(x)\leq(h(x)\mu(x))^{-1/\alpha}\leq (h\mu)_m^{-1/\alpha}\vee1=c_1(\alpha,h,G)$. Let $|V|$ be the number of all vertices, then from
$$1=\int_Vhu_p^{\alpha}d\mu\leq(h\mu)_M|V|\max_{x\in V}u_p^\alpha(x)$$
we see $\max_{x\in V}u_p(x)\geq(h_M\mu_M|V|)^{-1/\alpha}\wedge1=c_2(\alpha,h,G)$.
\end{proof}

\begin{lemma}\label{thm-lamda-p-bound}
There are positive constants $c_1(\alpha,g,h,G)\geq1$ and $c_2(\alpha,g,h,G)\leq1$ so that
\begin{equation}\label{formula-lamda-p-bound}
c_2(\alpha,g,h,G)\leq\lambda_p\leq c_1(\alpha,g,h,G).
\end{equation}
\end{lemma}
\begin{proof}
Assume $u_p$ attains its maximum at $x_0\in V$, then $\Delta_pu_p(x_0)\leq0$ by the definition of $\Delta_p$. From (\ref{u-p-solution}), we have $-\Delta_pu_p(x_0)+g(x_0)u_p^{p-1}(x_0)=\lambda_p h(x_0)u_p^{\alpha-1}(x_0)$. Hence
\begin{equation*}
\begin{aligned}
\lambda_p=\;&\frac{-\Delta_p u_p(x_0)+g(x_0)u_p^{p-1}(x_0)}{h(x_0)u_p^{\alpha-1}(x_0)}\\
\geq\;&g(x_0)h(x_0)^{-1}u_p^{p-\alpha}(x_0)\\
\geq\;&(gh^{-1})_mc_1(\alpha,h,G)^{1-\alpha}\wedge1\\
=\;&c_2(\alpha,g,h,G),
\end{aligned}
\end{equation*}
where we have used $u_p^{p-\alpha}(x_0)\geq c_1(\alpha,h,G)^{p-\alpha}\geq c_1(\alpha,h,G)^{1-\alpha}$ in the last inequality.

Similarly, from
\begin{equation*}
\begin{aligned}
|\Delta_p u_p(x_0)|\leq\;&\frac{1}{\mu(x_0)}\sum\limits_{y\thicksim x_0}w_{x_0y}|u_p(y)-u_p(x_0)|^{p-1}\\
\leq\;&\frac{1}{\mu(x_0)}\sum\limits_{y\thicksim x_0}w_{x_0y}\big(2u_p(x_0)\big)^{p-1}\\
\leq\;&c(\alpha,G)u_p^{p-1}(x_0)
\end{aligned}
\end{equation*}
we obtain
\begin{equation*}
\begin{aligned}
\lambda_p=\;&\frac{-\Delta_p u_p(x_0)+g(x_0)u_p^{p-1}(x_0)}{h(x_0)u_p^{\alpha-1}(x_0)}\\
\leq\;& c(\alpha,g,h,G)u_p^{p-\alpha}(x_0)\\
\leq\;& c(\alpha,g,h,G)c_2(\alpha,h,G)^{1-\alpha}\vee1\\
=\;&c_1(\alpha,g,h,G)
\end{aligned}
\end{equation*}
where we have used $u_p^{p-\alpha}(x_0)\leq c_2(\alpha,h,G)^{p-\alpha}\leq c_2(\alpha,h,G)^{1-\alpha}$ in the last inequality.
\end{proof}

\begin{lemma}\label{thm-hat-u-p-uniform-bound}
For every $p\in(1,\frac{\alpha+1}{2})$, the following equation
\begin{equation}\label{equ-hat-u-p}
-\Delta_p\hat{u}_p+g\hat{u}_p^{p-1}=h\hat{u}_p^{\alpha-1}
\end{equation}
has a positive solution $\hat{u}_p$ with $c_4(\alpha,g,h,G)\leq\hat{u}_p\leq c_3(\alpha,g,h,G)$.
\end{lemma}
\begin{proof}
Set $\hat{u}_p=u_p\lambda_p^{\frac{1}{\alpha-p}}$, then it is easy to see $\hat{u}_p$ is a positive solution of (\ref{equ-hat-u-p}). From the estimates (\ref{formula-lamda-p-bound}), we get
$$c_2^{\frac{2}{\alpha+1}}\leq c_2(\alpha,g,h,G)^{\frac{1}{\alpha-p}}\leq\lambda_p^{\frac{1}{\alpha-p}}\leq c_1(\alpha,g,h,G)^{\frac{1}{\alpha-p}}\leq c_1^{\frac{2}{\alpha-1}}.$$
Combining with the estimate (\ref{formula-up-bound}), we get the conclusion.
\end{proof}

\section{Proof of Theorem \ref{main-thm-the-paper}}
Since $G$ is finite graph, we can choose a function $u\in C(V)$ and a sequence $p_n\downarrow 1$, so that $\hat{u}_{p_n}\rightarrow u$. Since $\max\hat{u}_p$ is uniformly bounded by Lemma \ref{thm-up-bound} and Lemma \ref{thm-hat-u-p-uniform-bound}, we see $u\geq0$ and $u\neq0$. We can always choose a subsequence of $p_n$, which is still denoted as $\{p_n\}$ itself, a function $\xi\in C(V)$ and an edge weight $\eta$ defined on $E$, so that at each vertex $x\in V$,
\begin{enumerate}
  \item $\hat{u}_{p_n}(x)^{\alpha-1}\rightarrow\hat{u}(x)^{\alpha-1}$.
  \item $\hat{u}_{p_n}(x)^{p_n-1}\rightarrow \xi(x)\in[0,1]$.
  \item $|\hat{u}_{p_n}(y)-\hat{u}_{p_n}(x)|^{p_n-2}(\hat{u}_{p_n}(y)-\hat{u}_{p_n}(x))\rightarrow \eta(x,y)\in[-1,1]$, for $y\thicksim x$.
\end{enumerate}
It is easy to see $-\frac{1}{\mu(x)}\sum\limits_{y\thicksim x}w_{xy}\eta(x,y)+g(x)\xi(x)\in A^+(x)$ and $h(x)u(x)^{\alpha-1}\in A^-(x)$. Observe the equality
\begin{equation*}
-\Delta_p\hat{u}_p(x)+g(x)\hat{u}_p(x)^{p-1}=h(x)\hat{u}_p(x)^{\alpha-1},
\end{equation*}
and let $p_n\rightarrow1$, we obtain
$$-\frac{1}{\mu(x)}\sum\limits_{y\thicksim x}w_{xy}\eta(x,y)+g(x)\xi(x)=h(x)u(x)^{\alpha-1}.$$
This shows $A^+(x)\cap A^-(x)\neq\emptyset$ for every vertex $x\in V$. That is Theorem \ref{main-thm-the-paper}.

Huabin Ge: hbge@bjtu.edu.cn

Department of Mathematics, Beijing Jiaotong University, Beijing, China\\

Wenfeng Jiang: wen\_feng1912@outlook.com

School of Mathematics (Zhuhai), Sun Yat-Sen University, Zhuhai, China.


\begin{thebibliography}{50}
\setlength{\itemsep}{-2pt} \small

\bibitem{Aubin} T. Aubin, \emph{Some nonlinear problems in Riemannian geometry}, Springer Monographs in Mathematics. Springer-Verlag, Berlin, 1998.

\bibitem{ZGQ-ccm} K. C. Chang, \emph{The spectrum of the 1-Laplace operator}, Commun. Contemp. Math. 11 (2009), no. 5, 865-894.

\bibitem{ZGQ-JGraph} K. C. Chang, \emph{Spectrum of the 1-Laplacian and Cheeger's constant on graphs}, J. Graph Theory 81 (2016), no. 2, 167-207.

\bibitem{ZGQ-Advance} K. C. Chang, S. H. Shao, D. Zhang, \emph{Dong Nodal domains of eigenvectors for 1-Laplacian on graphs}, Adv. Math. 308 (2017), 529-574.

\bibitem{ZGQ-JComput} K. C. Chang, S. H. Shao, D. Zhang, \emph{Dong The 1-Laplacian Cheeger cut: theory and algorithms}, J. Comput. Math. 33 (2015), no. 5, 443-467.

\bibitem{Ge-jmaa} H. B. Ge, \emph{Kazdan-Warner equation on graph in the negative case}, J. Math. Anal. Appl.  453  (2017),  no. 2, 1022-1027.

\bibitem{Ge-pro-ams} H. B. Ge, \emph{A $p$-th Yamabe equation on graph}, Proc. Amer. Math. Soc., in press. DOI: https://doi.org/10.1090/proc/13929.

\bibitem{Ge-arxiv} H. B. Ge, \emph{p-th Kazdan Warner equation on graph}, preprint, arXiv:1611.04902.

\bibitem{Ge-Jiang-CCM} H. B. Ge, W. F. Jiang, \emph{p-th Yamabe equation on infinite graphs}, submitted to Commun. Contemp. Math..

\bibitem{Ge-Jiang-arxiv} H. B. Ge, W. F. Jiang, \emph{Kazdan-Warner equation on infinite graphs}, preprint, arXiv:1706.08698.

\bibitem{GLY} A. Grigor'yan, Y. Lin, Y. Y. Yang, \emph{Kazdan-Warner equation on graph}, Calc. Var. Partial Differential Equations 55 (2016), no. 4, Paper No. 92, 13 pp.

\bibitem{GLY'} A. Grigor'yan, Y. Lin, Y. Y. Yang, \emph{Yamabe type equations on graphs}, J. Differential Equations 261 (2016), no. 9, 4924-4943.

\bibitem{GLY3} A. Grigor'yan, Y. Lin, Y. Y. Yang, \emph{Existence of positive solutions to some nonlinear equations on locally finite graphs}, Sci. China Math.  60  (2017),  no. 7, 1311-1324.

\bibitem{KZ-2} J. L. Kazdan, F. W. Warner, \emph{Curvature functions for open 2-manifolds}, Ann. of Math. (2) 99 (1974), 203-219.

\bibitem{KZ-3} J. L. Kazdan, F. W. Warner, \emph{Existence and conformal deformation of metrics with prescribed Gaussian and scalar curvatures}, Ann. of Math. (2) 101 (1975), 317-331.

\bibitem{KZ-4} J. L. Kazdan, F. W. Warner, \emph{Scalar curvature and conformal deformation of Riemannian structure}, J. Differ. Geom. 10 (1975), 113-134.

\bibitem{Keller} M. Keller, M. Schwarz, \emph{The Kazdan-Warner equation on canonically compactifiable graphs}, preprint, arXiv:1707.08318.

\bibitem{KF}Kawohl, B., Fridman, V. \emph{Isoperimetric estimates for the first eigenvalue of the p-Laplace  operator and the Cheeger constant}, Comment. Math. Univ. Carolinae, 44 (2003), 659-667.

\bibitem{KS}Kawohl, B., Schuricht, F., Dirichlet, \emph{Problems for the 1-Laplace operator, including the eigenvalue problem,} Comm. in Contemporary Math. V0l. 9, (2007), 515-544.

\bibitem{Lee} J. M. Lee, T. H. Parker, \emph{The Yamabe problem}, Bull. Amer. Math. Soc. (N.S.) 17 (1987), no. 1, 37-91.

\bibitem{MRS} P. Mastrolia, M. Rigoli, A. G. Setti, \emph{Yamabe-type equations on complete, noncompact manifolds}, Progress in Mathematics, 302. Birkh\"{a}user/Springer Basel AG, Basel, 2012.

\bibitem{S} R. Schoen, \emph{Conformal deformation of a Riemannian metric to constant scalar curvature}, J. Differ. Geom. 20 (1984), 479-495.

\bibitem{T} N. Trudinger, \emph{Remarks concerning the conformal deformation of Riemannian structures on compact manifolds}, Ann. Scuola Norm. Sup. Pisa. 3 (1968), 265-274.

\bibitem{Y} H. Yamabe, \emph{On a deformation of Riemannian structurs on compact manifolds}, Osaka Math. J. 12 (1960), 21-37.

\bibitem{Zhang-Chang-jmaa} X. X. Zhang, Y. X. Chang, \emph{$p$-th Kazdan-Warner equation on graph in the negative case}, 2017, submitted to J. Math. Anal. Appl..

\bibitem{ZhangLin-arxiv} X. X. Zhang, A. J. Lin, \emph{Positive solutions of p-th Yamabe type equations on graphs}, preprint, arXiv:1708.07092
\end{thebibliography}
\end{document}